\documentclass[12pt,a4]{amsart}
\usepackage{a4}
\usepackage{amsmath}
\usepackage{amssymb}
\usepackage{graphicx}
\def\+{\oplus}

\newcommand{\G}{\Gamma}

\newcommand{\R}{{\mathbb R}}

\newcommand{\N}{{\mathbb N}}

\newcommand{\cS}{{\mathcal S}}

\newcommand{\cT}{{\mathcal T}}

\renewcommand{\phi}{\varphi}
\renewcommand{\a}{{\alpha}}

\newcommand{\g}{{\gamma}}
\renewcommand{\d}{{\delta}}
\newcommand{\e}{\epsilon}

\newcommand{\pd}{\partial}

\def\squareforqed{\hbox{\rlap{$\sqcap$}$\sqcup$}}
\def\qed{\ifmmode\else\unskip\quad\fi\squareforqed}
\def\smartqed{\def\qed{\ifmmode\squareforqed\else{\unskip\nobreak\hfil
\penalty50\hskip1em\null\nobreak\hfil\squareforqed
\parfillskip=0pt\finalhyphendemerits=0\endgraf}\fi}}

\newtheorem{remark}{\textbf{Remark}}[section]

\newtheorem{theorem}{\textbf{Theorem}}[section]

\newtheorem{proposition}{\textbf{Proposition}}[section]

\newtheorem{definition}{\textbf{Definition}}[section]
\newtheorem{acknowledgement}{\textbf{Acknowledgement}}

\numberwithin{equation}{section}
\begin{document}
\title[Eikonal equation on networks]{Viscosity solutions of Eikonal equations on
topological networks}
\author{Dirk Schieborn}
\address{Eberhard-Karls University,  T\"ubingen, Germany  (e-mail:Dirk@schieborn.de)}
\author{Fabio Camilli}
\address{Dipartimento di Scienze di Base e Applicate per l'Ingegneria,  ``Sapienza" Universit{\`a}  di Roma,
 00161 Roma, Italy, (e-mail:camilli@dmmm.uniroma1.it)}
\date{\today}

\begin{abstract}
In this paper we introduce a notion of viscosity solutions for Eikonal equations defined on
topological networks. Existence of a solution for the Dirichlet problem is obtained  via
representation formulas involving a distance function associated to the Hamiltonian.  A
comparison theorem based on Ishii's classical argument yields the uniqueness of the solution.
\end{abstract}

% AMS subject classifications (used in AMS journals)

% AMS keywords (used in AMS journals)

\maketitle

\section{Introduction}
Several  phenomena in
physics, chemistry and biology, described by  interaction of different media, can be translated into
mathematical problems involving differential equations which are not defined on connected
manifolds as usual, but instead on so-called ramified spaces. The latter can be roughly
visualized as a collection of different manifolds of the same dimension (branches) with certain
parts of their boundaries identified (ramification space). The simplest examples of
ramified spaces are \emph{topological networks}, which basically are graphs embedded in Euclidean space.
The interaction among the collection of differential equations describing the behavior
of physical quantities on the branches is described by certain transition
conditions governing the interaction of the quantities across the ramification spaces. \par
From a mathematical point of view, the concept of ramified spaces has originally been introduced by  Lumer \cite{lu}
and has later been refined and specified by various authors, e. g., J. von Below and S. Nicaise
\cite{ni}. Since 1980, many results have been published treating
different kinds of interaction problems involving linear and quasi-linear differential equations
(confer for instance Lagnese and Leugering \cite{lal}, Lagnese, Leugering, and
Schmidt \cite{lls}, von Below and Nicaise \cite{vbn}).
\par
As far as we know, fully nonlinear equations such as Hamilton-Jacobi
equations have not yet been examined to a similar extent on ramified spaces.
In the present paper we attack the problem by extending the theory of viscosity solutions
for Hamilton-Jacobi equations  to topological networks.
The major task in this context is to establish the correct transition conditions   the  viscosity solutions
are subjected to at transition vertices. As a matter of fact, these transition conditions
make up the core of our theory, as they constitute the major difference
from the classical theory of viscosity solutions.\par
The main result of the present paper consists in the observation that the concept of viscosity
solutions can indeed be appropriately extended to the class of
first order Hamilton-Jacobi equations of Eikonal type on a topological network.
Any generalization of existing concepts to new scenarios has to be justified
by the preservation of essential features. In
the case of the theory of viscosity solutions, these features are uniqueness, existence,
and stability. We will show that our generalization
of viscosity solutions to networks will be just as ``weak'' to yield existence, while being
sufficiently ``selective'' in order to ensure uniqueness and stability with respect to uniform
convergence. We will also demonstrate that our
definition arises as a natural selection principle, which in particular selects the distance function as the unique
viscosity solution of the Dirichlet problem for the standard Eikonal equation on networks. \par
A different attempt to study Hamilton-Jacobi equations on networks has already been made in \cite{acct}. However, the aim of this paper
deviates from the one addressed in the present paper: its main issue is to characterize the value function of controlled dynamics in
$\R^2$ restricted to a network. Therefore, the choice of the Hamiltonian, which
may be discontinuous with respect to the state variable, has to be restricted by assumptions ensuring both a suitable continuity
property with respect to the state variable and the fact that the set of admissible controls be not empty at any point of the network.
Additionally, the definition of viscosity solution characterizing the value function is different from our approach,
as it involves directional derivatives
of test functions in $\R^2$ along the edges.
In the present paper, Hamilton-Jacobi equations and differentiation along the edges
 are given in an intrinsic way making use of the maps
 embedding the network in $\R^N$, hence the approach is intrinsically 1-dimensional. Moreover in our approach
appropriate assumptions at the transition vertices guarantee the continuity of the Hamiltonian  with respect to the state variable.\par The existence of a viscosity solution is obtained by
a representation formula involving a distance associated to the Hamiltonian (see \cite{cs2}, \cite{fs}, \cite{i2} for corresponding results on
connected domains), the solution turning out to be the maximal subsolution
of the problem. Uniqueness, on the other hand, relies on a comparison principle inspired by Ishii's classical argument for Eikonal equations \cite{i3}.
In this respect, the existence of a strict subsolution plays a key role.\par
An important and classical problem in graph theory is the \emph{shortest path problem}, i.e. the problem of  computing in a weighted graph the distance of the  vertices from a given target vertex
(\cite{bel}). The weights  represent the cost of running through the edges.
 A motivation of our work is to generalize the previous  problem to the case of a running cost  which varies in a continuous way along the edges. In this case the aim is  to compute the distance of any point of the graph from a given
target set and this  in practice corresponds to solve the Eikonal equation $|Du|=\alpha(x)$ on the network with a zero-boundary condition on the
target vertices. Moreover
Hamilton-Jacobi equations of Eikonal type are important in several fields, for example geometric optics \cite{be}, homogenization \cite{e1,lpv}, singular perturbation \cite{ab}, weak KAM theory \cite{e2,fa,fs}, large-time behavior \cite{i}, and mean field games theory \cite{ll}.\par
The paper is organized as follows: In section \ref{sec1} we introduce the definitions of topological networks and viscosity solutions.
In section \ref{sec2} we collect some basic properties of viscosity solutions, in particular stability with respect to uniform convergence.
Section \ref{sec3} is devoted to the study of a distance function associated to the Hamilton-Jacobi equation, while section \ref{sec4} presents
the proof of a comparison principle. In section \ref{sec5} the representation formula for the solution of the Dirichlet problem is given.

\begin{acknowledgement}
This work is based on earlier results contained in the Ph.D. thesis of the first author. He would like to express his gratitude to Prof. K. P. Hadeler  for the support
and the guidance during the completion of the thesis.
\end{acknowledgement}
%%%%%%%%%%%%%%%%%%%%%%%%%%%%%%%%%%%%%%%%%%%%%%%%%%%
%                                                 %
%%%%%%%%%%%%%%%%%%%%%%%%%%%%%%%%%%%%%%%%%%%%%%%%%%%
\section{Assumptions and preliminary definitions}\label{sec1}
We start with the definition of a topological network.
\begin{definition}
Let $V=\{v_i,\,i\in I\}$ be a finite collection of pairwise different points in $\R^N$ and
let $\{\pi_j,\,j\in J\}$ be a finite collection of differentiable, non self-intersecting curves in $\R^N$ given by
\[\pi_j:[0,l_j]\to\R^N,\, l_j>0,\,j\in J.\]
Set $e_j:=\pi_j((0,l_j))$, $\bar e_j:=\pi_j([0,l_j])$, and $E:=\{e_j:\, j\in J\}$. Furthermore assume that
\begin{itemize}
  \item[i)] $\pi_j(0), \pi_j(l_j)\in V$ for all $j\in J$,
  \item[ii)]$\#(\bar e_j\cap V)=2$ for all $j\in J$,
  \item[iii)] $\bar e_j\cap \bar e_k\subset  V$, and $\#(\bar e_j\cap \bar e_k)\le 1$ for all $j,k \in J$, $j\neq k$.
  \item[iv)] For all $v, w \in V$ there is a path with end points  $v $ and $w$ (i.e.  a sequence of edges $\{e_j\}_{j=1}^N$ such that
  $\#(\bar e_j\cap \bar e_{j+1})=1$ and  $v\in \bar e_1$, $w\in \bar e_N$).
\end{itemize}
Then $\bar \G:=\bigcup_{j\in J}\bar e_j\subset \R^N$ is called a (finite)  \emph{topological network} in $\R^N$.
\end{definition}
If $v_i\in V\cap\bar e_j$ we say that $e_j$ is incident to $v_i$ ($ e_j\,\text{inc}\, v_i$ in short). For $i\in I$ we set
$Inc_i:=\{j\in J:\,e_j \,\text{inc}\,v_i\}$.
Observe that the parametrization of the arcs $e_j$ induces an orientation on the edges, which can be expressed
by the \emph{signed incidence matrix} $A=\{a_{ij}\}_{i,j\in J}$ with
\begin{equation}\label{1:1}
   a_{ij}:=\left\{
            \begin{array}{rl}
              1 & \hbox{if $v_i\in\bar e_j$ and $\pi_j(0)=v_i$,} \\
              -1 & \hbox{if $v_i\in\bar e_j$ and $\pi_j(l_j)=v_i$,} \\
              0 & \hbox{otherwise.}
            \end{array}
          \right.
\end{equation}
In the following we will study boundary value problems on $\bar\G$. Given a  nonempty set $I_B\subset I$, we define  $\partial \G:=\{v_i,\,i\in I_B\}$ to be the set of \emph{boundaries vertices}, while for   $I_T:=I\setminus I_B$   we call $\{v_i, \,i\in I_T\}$ the set of
\emph{transition vertices}.
We also set $\G:=\bar \G\setminus \partial\G$. We always assume $i\in I_B$ whenever $\#(Inc_i)=1$ for some $i\in I$.
We remark that in applications such as the  shortest path problem it is interesting to impose a boundary condition also
at the internal vertices, i.e. vertices with $\#(Inc_i)>1$.\par
 Consider the subspace topology induced to $\bar \G$ by $\R^N$. It coincides with the topology induced by the \emph{path distance}
\begin{equation}\label{1:dist}
   d(y,x):=\inf\left\{\int_0^t|\dot\g(s))|ds:\, t>0,\,\g\in B^t_{y,x}\right\}\qquad\text{for $x,y\in \bar\G$, where}
\end{equation}
\begin{itemize}
  \item[i)] $\g:[0,t]\to \G$ is a piecewise differentiable path in the sense that there are
$t_0:=0<t_1<\dots<t_{n+1}:=t$ such that for any $m=0,\dots,n$, we have $\g([t_m,t_{m+1}])\subset \bar e_{j_m}$ for some
$j_m\in J$, $\pi_{j_m}^{-1}\circ \g\in C^1(t_m,t_{m+1})$, and
\[\dot\g(s)=\frac{d}{ds}(\pi_{j_m}^{-1}\circ  \g)(s).\]
  \item[ii)]  $B^t_{y,x}$ is the set of all such paths with $\g(0)=y$, $\g(t)=x$.
\end{itemize}
For any function $u:\bar \G\to\R$ and each $j\in J$ we denote by $u^j$ the restriction of $u$ to $\bar e_j$, i.e.
\[u^j:=u\circ \pi_j:[0,l_j]\to \R.\]
We say that $u$ is continuous in $\bar\G$ and write $u\in C(\bar\G)$ if $u$ is continuous with respect to the subspace topology of $\bar\G$. This means
that $u^j\in C([0,l_j]$ for any $j\in J$ and
\[u^j(\pi_j^{-1}(v_i))=u^k(\pi_k^{-1}(v_i)) \qquad\text{for any $i\in I$, $j,k\in Inc_i$.}\]
In a similar way we define the space of upper semicontinuous functions  $\text{USC}(\bar\G)$ and the space of
lower semicontinuous functions $\text{LSC}(\bar\G)$, respectively, and the space $C(\Gamma)$.\\
We define differentiation along an edge $e_j$ by
\[\pd_ju(x):=\pd_j u^j(\pi_j^{-1}(x))=\frac{\pd}{\pd x} u^j(\pi_j^{-1}(x)),\qquad\text{for all $x\in e_j$,}\]
and at a vertex $v_i$ by
\[\pd_ju(v_i):=\pd_j u^j(\pi_j^{-1}(v_i))=\frac{\pd}{\pd x} u^j(\pi_j^{-1}(v_i))\qquad\text{for $j\in Inc_i$.}\]
%We write $u\in C^0(\bar\G)$ if $u$ and $\pd U$ are  continuous with respect to the subspace topology of $\bar\G$.\par
%We say that $u$ is differentiable if $u^j\in C^1([0,l_j])$ for any $j\in J$ and
%In this case we write $u\in C^1(\bar \G)$.
A Hamiltonian $H:\bar \G\times \R\to \R$ of \emph{eikonal type}  is a collection $(H^j)_{j\in J}$ with $H^j: [0,l_j]\times\R\to \R$
 satisfying the following conditions:
\begin{align}
    &H^j\in C^0([0,l_j]\times \R), \quad j\in J,\label{1:H1}\\
    &H^j(x,p)\quad \text{is convex in $p\in \R$ for any $x\in [0,l_j]$, $j\in J$,}\label{1:H2}\\
   &H^j(x,p)\to +\infty\quad \text{as $|p|\to \infty$ for any $x\in [0,l_j]$, $j\in J$, }\label{1:H3}\\
   &H^j(\pi_j^{-1}(v_i),p)=H^k(\pi_k^{-1}(v_i),p)\quad\text{for any $p\in \R$, $i\in I$, $j,k \in Inc_i$, }\label{1:H4}\\
   &H^j(\pi_j^{-1}(v_i),p)=H^j(\pi_j^{-1}(v_i),-p)\quad\text{for any $p\in \R$, $i\in I$, $j\in Inc_i$. }\label{1:H5}
\end{align}
\begin{remark}\label{1:rem1}
Assumptions \eqref{1:H1}--\eqref{1:H3} provide standard conditions in the theory of viscosity solutions (see \cite{fs}, \cite{i2}). Assumptions
\eqref{1:H4}--\eqref{1:H5}  represent reasonable compatibility conditions of $H$  at the vertices of $\bar \G$, i.e.  continuity at the
vertices and independence of the orientation of the incident arc, respectively (the network is not oriented).
A typical example of a Hamiltonian satisfying these assumptions is  $H^j(x,p):=p^2-\alpha (x)$, $j\in J$, where $\alpha(x)=\alpha^j(\pi_j^{-1}(x))$
for $x\in \bar e_j$ and  $\alpha^j\in C^0([0,l_j])$, $\alpha^j(x)\ge0$ for $x\in\bar\G$, $\alpha^j(\pi_j^{-1}(v_i))=\alpha^k(\pi_k^{-1}(v_i))$ for any $i\in I$, $j,k\in Inc_i$.
\end{remark}
\begin{definition}\label{1:def2}
Let $\phi\in C(\G)$. \begin{itemize}
  \item[i)] Let $x\in e_j$, $j\in J$. We say that $\phi$ is differentiable at $x$, if $\phi^j$ is differentiable at $\pi_j^{-1}(x)$.
  \item[ii)] Let $x=v_i$, $i\in I_T$, $j,k\in Inc_i$, $j\neq k$. We say that $\phi$ is $(j,k)$-differentiable at $x$, if
\begin{equation}\label{1:2}
   a_{ij}\pd_j \phi_j(\pi_j^{-1}(x))+a_{ik}\pd_k \phi_k(\pi_k^{-1}(x))=0,
\end{equation}
where $(a_{ij})$ as in \eqref{1:1}. Moreover, we say that $\phi$ is differentiable at $x$ if $\phi$ is $(j,k)$-differentiable  at $x$
for any $j,k\in Inc_i$, $j\neq k$. \end{itemize}

\end{definition}
\begin{remark}
Condition \eqref{1:2} demands that the derivatives in the direction of the incident edges $j$ and $k$ at  the vertex $v_i$ coincide,
 taking into account the orientation of the edges.
\end{remark}
On topological networks we now introduce the definition of viscosity solutions $u$ of Hamilton-Jacobi equations of eikonal type of the form
\begin{equation}\label{HJ}
    H(x,Du)=0,\qquad x\in \G.
\end{equation}
\begin{definition}\label{1:def3}\hfill\\
A function $u\in\text{USC}(\bar\G)$ is called a (viscosity) subsolution of \eqref{HJ} in $\G$ if the following holds:
\begin{itemize}
  \item[i)] For any  $x\in e_j$, $j\in J$, and for any $\phi\in C(\G)$ which is differentiable at $x$ and for which
  $u-\phi$ attains a local maximum at $x$, we have
 \[H^j(\pi_j^{-1}(x), \pd_j \phi_j(\pi_j^{-1}(x))\le 0.\]
  \item [ii)] For any $x=v_i$, $i\in I_T$, and for any $\phi$ which is $(j,k)$-differentiable at $x$ and for which
  $u-\phi$ attains a local maximum at $x$, we have
 \[H^j(\pi_j^{-1}(x), \pd_j \phi_j(\pi_j^{-1}(x))\le 0.\]
\end{itemize}
 A function $u\in\text{LSC}(\bar\G)$ is called a (viscosity) supersolution of \eqref{HJ} in $\G$ if the following holds:
\begin{itemize}
  \item[i)] For any  $x\in e_j$, $j\in J$,  and for any $\phi\in C(\G)$ which is differentiable at $x$ and for which $u-\phi$ attains a
  local minimum at $x$, we have
 \[H^j(\pi_j^{-1}(x), \pd_j \phi_j(\pi_j^{-1}(x))\ge 0.\]
  \item [ii)] For any $x=v_i$, $i\in I_T$, $j\in Inc_i$, there exists $k\in Inc_i$, $k\neq j$, (which we will call $i$-feasible for $j$ at $x$) such that for any $\phi\in C(\G)$ which is
  $(j,k)$-differentiable
  at $x$ and for which $u-\phi$ attains a local maximum at $x$, we have
 \[H^j(\pi_j^{-1}(x), \pd_j \phi_j(\pi_j^{-1}(x))\ge 0.\]
\end{itemize}
A continuous function $u\in C(\G)$ is called a (viscosity) solution of  \eqref{HJ} if it is both a viscosity subsolution and a viscosity supersolution.
\end{definition}

\begin{remark}\label{1:r2}
\emph{i)} Let $i\in I_T$ and  $\phi\in C(\G)$  be $(j,k)$-differentiable at $x$. Then by \eqref{1:H4}-\eqref{1:H5}, we have
\begin{equation}\label{1:r21}
H^j(\pi_j^{-1}(x), \pd_j \phi(\pi_j^{-1}(x)))=H^k(\pi_k^{-1}(x),\pm \pd_k \phi(\pi_k^{-1}(x)));
\end{equation}
hence in  the subsolution condition, it is indifferent to require the condition for $j$ or for $k$. \\
\emph{ii)} To simplify the notation, we set
\[H^j(x,\pd_j\phi(x)):=H^j(\pi_j^{-1}(x), \pd_j \phi(\pi_j^{-1}(x))).\]
Moreover, we will call $\phi\in C(\G)$ an upper (lower) $(j,k)$-test function of $u$ at $x=v_i$ if it is
$(j,k)$-differentiable at $x$ and if $u-\phi$ attains a local maximum (minimum) at $x$.\\
\emph{iii)} It is important to observe the asymmetry in definition \ref{1:def3} regarding the subsolution and the supersolution conditions
at the transition vertices. It reflects the idea that distance functions have to be solutions  of \eqref{HJ}
and that there is always a shortest path from a transition vertex to the boundary. In fact it is worthwhile to observe that if
supersolutions were defined similarly
to subsolutions, the conditions in general would not be satisfied by $\inf\{d(y,x):\,y\in \pd \G\}$, which is, as we will see in
section \ref{sec5}, the
solution of $|Du|^2-1=0$ with zero boundary conditions.\\
\emph{iv)} Taking \eqref{1:r21} into account, it is easily seen that a viscosity solution $u$ of \eqref{HJ} satisfies
the equation in a pointwise sense at any point $x\in \G$ where
it is differentiable.
\end{remark}
As the conditions in definition \ref{1:def3} are of pointwise character, they can also be imposed on subsets of $\G$. Hence,
for $\Omega \subseteq \G$ we denote by  $\cS (\Omega)$ ($\cS^+ (\Omega)$ or $\cS^- (\Omega)$, respectively)
the space of solutions (supersolutions or  subsolutions, respectively) of \eqref{HJ}in $\Omega$. For short, we set $\cS:=\cS(\G)$, $\cS^+:=\cS^+(\G)$, and $\cS^-:=\cS^-(\G)$.
%%%%%%%%%%%%%%%%%%%%%%%
%                     %
%%%%%%%%%%%%%%%%%%%%%%%
\section{Some basic properties of viscosity solutions}\label{sec2}
In this section we discuss some basic properties of the viscosity solutions introduced in the previous section.
\begin{proposition} \label{1:prop1}
Let $u$, $v$ be subsolutions (supersolutions) of \eqref{HJ} in $\G$.
Then $w :=\max\{u,v\}$ ($w :=\min\{u,v\}$) is a subsolution (a supersolution) of \eqref{HJ} in $\G$.
\end{proposition}
\begin{proof}
We only consider the case $x=v_i$, $i\in I$, as otherwise the argument is standard.\\
Let $u,v\in\text{USC}(\G)$ be two subsolutions at $x$ and observe that $w=\max\{u,v\}\in \text{USC}(\G)$. Let $j,k\in Inc_i$ and
let $\phi$ be an upper $(j,k)$-test function of $w$ at $x$. If $w(x)=u(x)$ (similarly in the other case) then $\phi$ is an
  upper $(j,k)$-test function of $u$ at $x$, implying
\[H^j(x,\pd_j \phi(x))\le 0.\]
Hence $w$ is a subsolution.\\
Let $u,v\in\text{LSC}(\G)$ be two supersolutions at $x$, whence $w=\min\{u,v\}\in LSC(\G)$. Assume $w(x)=u(x)$ (similarly in the other case).
Hence for any $j\in Inc_i$ there exists $k\in Inc_i$, $k\neq j$, such that for any
 lower $(j,k)$-test function $\phi$ of $u$ at $x$ we have
\[H^j(x,\pd_j \phi(x))\ge 0.\]
Hence $k$ is $i$-feasible for $j$ also with respect to $w$.
\end{proof}

%%%%%%%%%%%%%
%           %
%%%%%%%%%%%%%
\begin{proposition}\label{2:prop2}
Assume $H_n(x,p)\to H(x,p)$ uniformly for $n\to \infty$ (i.e. $H_n^j(\pi_j^{-1}(x),p)\to H^j(\pi_j^{-1}(x),p)$
uniformly for $(x,p)\in \bar e_j\times \R$
for any $j\in J$). For any $n\in \N$ let $u_n$ be a solution of
\begin{equation}\label{HJn}
   H_n(x, D u)=0,\qquad x\in \G,
\end{equation}
and assume $u_n\to u$ uniformly in $\G$ for $n\to\infty$. Then $u$ is  a solution of  \eqref{HJ}.
\end{proposition}
\begin{proof}
We treat the case $x=v_i$, $i\in I_T$, as otherwise the argument is standard (see \cite{bcd}).\\
We first prove that $u$ is a \emph{subsolution}.
 Choose any $j,k\in Inc_i$, $j\neq k$, along with
an upper $(j,k)$-test function $\phi$   of $u$ at $x$. Consider the auxiliary function $\phi_\d(y):=\phi (y)+\d d(x,y)^2$ for $\d>0$.
Observe that $\pd_m ( d(x,\cdot)^2)(\pi_m^{-1}(x))=0$ for    $m=j$ and $m=k$, hence $d(x,\cdot)^2$ is $(j,k)$ differentiable at $x$.
Then $\phi_\d$ is an upper $(j,k)$-test function  of $u$ at $x$ and there exists $r>0$ such that $u-\phi_\d$ attains a strict local
maximum w.r.t. $\bar B_r(x)$ at $x$, where $B_r(x):=\{y\in\G:\, d(x,y)<r\}$. Observe that $x$ is a strict maximum point
for $u-\phi_\d$ also in $\bar B:=\bar B_r(x)\cap (\bar e_j \cup\bar e_k)$. Now choose a sequence $\omega_n\to 0$ for $n\to\infty$ with
\begin{equation}\label{2:prop21}
   \sup_{\G}|u(x)-u_n(x)|\le \omega_n
\end{equation}
and let $y_n$ be a maximum point for $u_n-\phi_\delta$ in $\bar B$. Up to  a subsequence, $y_n\to z\in \bar B$. Moreover,
\[u(x)-\phi_\d(x)-\omega_n\le u_n(x)-\phi_\d(x)\le u_n(y_n)-\phi_\d(y_n)\le u(y_n)-\phi_\d(y_n)+\omega_n.\]
For $n\to \infty$, we get
$u(x)-\phi_\d(x)\le u(z)-\phi_\d(z).$
As $x$ is a strict maximum point, we conclude $x=z$. Invoking
\[u(x)+\phi_\d(y_n)-\phi_\d(x)-\omega_n\le u_n(y_n)\le u(y_n)+\omega_n\]
we altogether get
\begin{equation}\label{2:prop22}
\lim_{n\to\infty} y_n= x, \quad    \lim_{n\to\infty}u_n(y_n)=u(x)
\end{equation}
We distinguish two cases:\\
\emph{Case 1: $y_n\neq x$. } Then $y_n\in e_m$ with either $m=j$ or $m=k$. Since $u_n-\phi_\d$ attains a maximum at $y_n$ and as
$\phi_\d$ is differentiable at $y_n$ with $\pd_m \phi_\d(y_n)=\pd_m\phi(y_n)+2\d a_{im}d(x,y_n)$, we have
\begin{equation}\label{2:prop23}
   H^m_n( y_n,\pd_m \phi(y_n)+2\d a_{im}d(x,y_n))\le 0.
\end{equation}
with either $m=j$ or $m=k$.\\
\emph{Case 2: $y_n= x$. } Then $\pd_m \phi_\d(y_n)=\pd_m \phi(y_n)$ for $m=j$ and $m=k$ and therefore
\begin{equation}\label{2:prop24}
    H^j_n(y_n,\pd_j\phi_\d(y_n))\le 0.
\end{equation}
By \eqref{2:prop22}, \eqref{2:prop23}, \eqref{2:prop24} and recalling \eqref{1:r21} we get for $n\to\infty$
\[ H^j(x,\pd_j\phi(x))\le 0.\]\\
To show  that $u$ is a \emph{supersolution}, we
 assume by contradiction that  there exists $j\in Inc_i$ such that for any $k\in Inc_i$, $k\neq j$, there exists
 a lower $(j,k)$-test function $\phi_k$ of $u$ at $x$ for which
 \begin{equation}\label{2:prop25}
    H^j(x,\pd_j\phi_k(x))<0.
 \end{equation}
%By the assumptions on $H$, we have
%\begin{equation}\label{2:prop26}
%    H^j(x,\pd_j\phi_k(x))= H^k(x,\pd_k\phi_k(x))<0
%\end{equation}
By adding a quadratic function of the form $-\a_kd(x,y)^2$ to the function $\phi_k$ we may assume that
there exists $r>0$ such that $u-\phi_k$ attains a strict minimum in $\bar B_r(x)$ at $x$. Observe that $x$ is a strict minimum point
of $u-\phi_k$ also in $\bar B_k:=\bar B_r(x)\cap (\bar e_j \cup\bar e_k)$.
Now for any $n\in\N$ there exists $k_n\in Inc_i$, $k_n\neq j$, which is $i$-feasible for  j with respect to $u_n$. Up to a subsequence,
we may assume that there exists $k\in Inc_i$ such that $k_n=k$ for any $n$.\\
 Let $y_n$ be a minimum point of $u_n-\phi_k$ in $\bar B_k$ and let $\omega_n$ be as in \eqref{2:prop21}.
 Similarly to the subsolution case, we can prove that \eqref{2:prop22} holds.
If $y_n\neq x$, we obtain
\begin{equation}\label{2:prop27}
   H^m_n( y_n,\pd_j \phi_k(y_n))\ge 0
\end{equation}
for  either $m=j$ or $m=k$. If $y_n= x$, we get
\begin{equation}\label{2:prop28}
    H^j_n(y_n,\pd_j\phi_k(y_n))\ge 0.
\end{equation}
Hence by \eqref{2:prop22}, \eqref{2:prop27}, and \eqref{2:prop28}, we get for $n\to \infty$
\[ H^j(x,\pd_j\phi_k(x))\ge 0,\]
which is a contradiction to \eqref{2:prop25}.
\end{proof}
The proof of the following proposition is given in \cite[Prop.II.4.1]{bcd}, for example.
\begin{proposition}
Let $K$ be a compact subset of $\G$ and let $u\in \cS^-(K)$. Then there exists a constant $C_K$ depending only on $K$ such that
\begin{equation}\label{2:prop1}
|u(x)-u(y)|\le C_Kd(x,y).
\end{equation}
\end{proposition}
%%%%%%%%%%%%%%5
%            %
%%%%%%%%%%%%%%
The proof of the next two propositions is very similar to the one of Prop.\ref{2:prop2}.
\begin{proposition}\label{2:prop3bis}
Let $\cT\subset \cS^- $ ($\cT\subset \cS^+ $) and set $u(x):=\sup\{v(x)|\,v\in \cT\}$ ($u(x):=\inf\{v(x)|\,v\in \cT\}$) for $x\in \G$.
Suppose that $u\in C(\G)$. Then $u\in \cS^- $ ($u\in \cS^+$).
\end{proposition}
\begin{proof}
To prove that $u(x)=\sup\{v(x)|\,v\in \cT\}$ is a subsolution, we only consider the case $x=v_i$, $i\in I_T$. Consider $j,k\in Inc_i$, $j\neq k$, and
an upper $(j,k)$-test function $\phi$   of $u$ at $x$. Set $\phi_\d(y):=\phi (y)+\d d(x,y)^2$ for $\d>0$.
Then $\phi_\d$ is an upper $(j,k)$-test function  of $u$ at $x$ and there exists $r>0$ such that $u-\phi_\d$ has
a strict local maximum point in $\bar B_r(x)$ at $x$. Observe that $x$ is a strict maximum point
for $u-\phi_\d$ also in $\bar B:=\bar B_r(x)\cap (\bar e_j \cup\bar e_k)$. Let $u_n\in S$ be such that
\[
   u(x)-u_n(x)\le \frac{1}{n}
\]
and let $y_n$ be a maximum point for $u_n-\phi_\delta$ in $\bar B$. Up to  a subsequence, $y_n\to z\in \bar B$. Moreover,
\[u(x)-\phi_\d(x)-\frac{1}{n}\le u_n(x)-\phi_\d(x)\le u_n(y_n)-\phi_\d(y_n)\le u(y_n)-\phi_\d(y_n).\]
For $n\to \infty$, we obtain $u(x)-\phi_\d(x)\le u(z)-\phi_\d(z)$, implying $x=z$, as $x$ is a strict maximum point. Moreover, by
\[u(x)+\phi_\d(y_n)-\phi_\d(x)\le u_n(y_n)\le u(y_n)\]
we get
\[
\lim_{n\to\infty} y_n= x, \quad    \lim_{n\to\infty}u_n(y_n)=u(x),
\] and we conclude as in Proposition \ref{2:prop2}.\\
Similarly to the proof of Proposition \ref{2:prop2} one can also show that $u(x):=\inf\{v(x)|\,v\in \cT\}$ is a supersolution.
\end{proof}
%%%%%%%%%%%%%%
%            %
%%%%%%%%%%%%%%
\begin{proposition}\label{2:prop3}
Let $\cT\subset \cS$ and let $u(x):=\inf\{v(x)|\,v\in \cT\}$ for $x\in \G$. Assume that $u(x)\in \R$ for some $x\in\G$. Then $u\in \cS$.
\end{proposition}
\begin{proof}
By  \eqref{2:prop1} all $v\in \cS$ are uniformly Lipschitz continuous. As $u(x)\in \R$, we thus have
$u(y)\in\R$ for any $y\in\G$. Moreover, $u$ is Lipschitz continuous on $\G$. Next observe that by Proposition \ref{2:prop3bis} $u$ is a supersolution
of \eqref{HJ}.\\	
 In order to prove that $u$ is also a subsolution we once more invoke (and only sketch) the argument used in Proposition \ref{2:prop2}.
Consider  $x=v_i$, $i\in I_T$,  $j,k\in Inc_i$, $j\neq k$, and
an upper $(j,k)$-test function $\phi$   of $u$ at $x$. Define the auxiliary function $\phi_\d(y):=\phi (y)+\d d(x,y)^2$ for $\d>0$.
Then $\phi_\d$ is an upper $(j,k)$-test function  of $u$ at $x$ and there exists $r>0$ such that $u-\phi_\d$ has a strict local maximum point in $\bar B_r(x)$ at $x$. Observe that $x$ is a strict maximum point
for $u-\phi_\d$ also in $\bar B:=\bar B_r(x)\cap (\bar e_j \cup\bar e_k)$. Let $u_n\in \cT$ be such that
\[
   u(x)-u_n(x)\ge -\frac{1}{n}
\]
and let $y_n$ be a maximum point for $u_n-\phi_\delta$ in $\bar B$. Up to  a subsequence, $y_n\to z\in \bar B$. Moreover,
\[u(x)-\phi_\d(x)\le u_n(x)-\phi_\d(x)\le u_n(y_n)-\phi_\d(y_n)\le u(y_n)-\phi_\d(y_n)+\frac{1}{n}.\]
Hence we obtain \eqref{2:prop22}.
Arguing as in Proposition \ref{2:prop2} we conclude
$H^j(x,\pd_j\phi(x))\le 0.$
\end{proof}

%%%%%%%%%%%%%%%%%%%%%%%%%%%%%%%%%%%%%%%%%%%%%%%%%%%
%                                                 %
%%%%%%%%%%%%%%%%%%%%%%%%%%%%%%%%%%%%%%%%%%%%%%%%%%%
\section{A distance function for Hamilton-Jacobi equations}\label{sec3}
In this section we assume
\begin{equation}\label{H4}
    \cS^-(\G)\neq \emptyset,
\end{equation}
i.e. there exists a subsolution of \eqref{HJ} in $\G$.
We  introduce  a distance function related to the Hamiltonian $H$ on the network.
For $x,y\in\G$ define
\begin{equation}\label{2:dist_int}
   S(y,x)=\inf\left\{\int_0^t L (\g(s),\dot\g(s))ds:\,t>0\,\g\in B^t_{y,x}\right\},
\end{equation}
where $B^t_{y,x}$ as in \eqref{1:dist} and
\[L^j(x,q):=\sup_{p\in\R}\{p\, q- H^j(x,p)\}=\sup_{p\in\R}\{p\, q- H^j(\pi_j^{-1}(x),p)\}\]
for any $j\in J$, $x\in \bar e_j$.
Note that the  distance defined by \eqref{2:dist_int} coincides with  the distance defined by \eqref{1:dist}   for  $H(x,p)=|p|^2-1$.
The next proposition summarizes some properties of $S$.
%%%%
\begin{proposition}\label{2:prop4}
$S$ is a  Lipschitz continuous distance on $\G\times \G$. Moreover, \begin{itemize}
\item[i)] for any $y\in \G$ we have $S(y,\cdot)\in \cS^-(\G)\cap \cS(\G\setminus\{y\})$,
\item[ii)]  for any $x,y\in \G$ we have
\begin{equation}\label{2:prop41}
S(y,x)=\max\{u(x)|\, u\in \cS^-(\G)\,\text{s.t. } u(y)=0\}.
\end{equation}
\end{itemize}
\end{proposition}
\begin{proof}
%We now prove \eqref{2:prop41}, i.e. that $S(y,\cdot)$ is the maximal subsolution vanishing at $y$.
By \eqref{2:prop1}  any subsolution $u$ of \eqref{HJ} in $\G$ is Lipschitz continuous. Integrating along a path joining $x$ and $y$ we get
\begin{equation}\label{4:lemma31}
   u(x)-u(y)\le S(y,x)\quad\text{for any $x,y\in\G$}.
\end{equation}
Thus by \eqref{H4} $S(y,x)$ is finite for any $x,y\in\bar\G$.

By the coercitivity  of $H$ assumed in \eqref{1:H3} there exists   constant $R>0$, $M$ such that
$L(x,q)\le M$  (i.e $L^j(\pi_j^{-1}(x),q)\le M$) for any $x\in\G$, $q\in B(0,R)$,  and therefore  $S(y,x)\le C_R\, d(y,x)$
(see  \cite[Prop.5.1]{i} for details). Moreover, given $x,y,z\in\G$,
the juxtaposition of two curves in $B_{y,z}$ and $B_{z,x}$ gives a curve in  $B_{y,x}$, whence
\[S(y,x)\le S(y,z)+S(z,x).\]
\indent\emph{$S(y,\cdot)$ is a subsolution: }
In order to prove that $S(y,\cdot)$ is a subsolution at $x_0\in\G$ we distinguish two cases:\\

\emph{Case 1:} $x_0\not \in \{v_i,i\in  I_T\}$. Assume $x_0\in e_j$ for some $j\in J$ and let $\psi$ be
an upper test function of $S(y,\cdot)$
at $x_0$. It follows $S(y,x_0)-\psi_j(\pi_j^{-1}(x_0))\ge S(y,x )-\psi_j(\pi_j^{-1}(x ))$ for $x\in B_r(x_0)\cap e_j$. Set
$t_0:=\pi_j^{-1}(x_0)$, fix $q\in \R$, and choose
  $h$ sufficiently small in such a way that $t_0-hq\in (0,l_j)$. Define
the curve $\g_h:[0,h]\to \G$ by $\g_h(s):=\pi_j(\frac{s}{h}t_0+(1-\frac{s}{h})t_h)$, where   $t_h:= t_0-hq $ and set $x_{hq}:=\pi_j(t_0-hq)$. Hence
\begin{align*}
& \pd \psi(x_0)\,q=   \pd_{j}\psi_j(t_0)\,q=\lim_{h\to 0^+}\frac{\psi_j(t_0)-\psi_j(t_h)}{h}\le \lim_{h\to 0^+}\frac{S(y,x_0)-S(y,x_{hq})}{h}\\
& \le \lim_{h\to 0^+}\frac{S(x_{hq},x_0)}{h}\le \lim_{h\to 0^+}\frac{1}{h}\int_0^h L(\g_h(s),\dot\g_h(s))ds=\\
&  \lim_{h\to 0^+}\frac{1}{h}\int_0^h L^j(\frac{s}{h}t_0+(1-\frac{s}{h})t_h,q)ds
=L^j(t_0,q)=L(x_0,q).
\end{align*}
Hence $H(x,\pd \psi(x_0))=\sup_{q}\{\pd \psi(x_0)\,q-L(x_0,q)\}\le 0$.\\

\emph{Case 2:} $x_0 \in \{v_i,i\in  I_T\}$. Assume $x_0=v_i$  and let $\psi$ be an upper $(j,k)$-test function of $S(y,\cdot)$ at $x_0$.
Set $t_m:=\pi_m^{-1}(x_0)$, $m=j,k$, and observe that for any $q\in\R$ and for $h$ sufficiently small we have $t_m-h(-a_{im}q)\in (0,l_m)$ (or equivalently
$\pi_m(t_m-h(-a_{im}q))\in e_m$) for $m=j,k$. Arguing as in case 1, we get
\begin{equation}\label{4:lemma10}
\begin{split}
    \pd_j\psi_j(\pi_j^{-1}(x_0))(-a_{ij}q)\le L^j(\pi_j^{-1}(x_0), -a_{ij}q),\\
     \pd_k\psi_k(\pi_k^{-1}(x_0))(-a_{ik}q)\le L^k(\pi_k^{-1}(x_0), -a_{ik}q).
    \end{split}
\end{equation}
Moreover, since $a_{ij}\pd_j\psi(\pi_j^{-1}(x_0))+a_{ik}\pd_k\psi_k(\pi_k^{-1}(x_0))=0$ and
$L^j(\pi_j^{-1}(x_0),q)=L^k(\pi_k^{-1}(x_0),q)$, we get
\begin{equation}\label{4:lemma10B}
  \begin{split}
     \pd_j\psi_j(\pi_j^{-1}(x_0))(a_{ij}q)\le L^j(\pi_j^{-1}(x_0), a_{ij}q),\\
     \pd_k\psi_k(\pi_k^{-1}(x_0))(a_{ik}q)\le L^k(\pi_k^{-1}(x_0), a_{ik}q).
    \end{split}
\end{equation}
By \eqref{4:lemma10} and \eqref{4:lemma10B} it follows $\pd_m\psi_m(\pi_m^{-1}(x_0))q\le L^m((\pi_m^{-1}(x_0),q)$ for $m=j,k$, whence
$H(x,\pd \psi(x_0))=\sup_{q}\{\pd \psi(x_0)q-L(x_0,q)\}\le 0$.
By \eqref{4:lemma31} and  since $S(y,\cdot )$ is a subsolution with $S(y,y)=0$, we finally obtain \eqref{2:prop41}.\vskip 4pt
\indent\emph{$S(y,\cdot)$ is a supersolution in $\G\setminus\{y\}$:}
In order to prove that $S(y, \cdot)$   is a supersolution at $x\neq y$ we only consider the case $x=v_i$, $i \in I_T$, being the other case standard.
Assume that $u(\cdot)=S(y, \cdot)$ is not a supersolution at $x$.
By definition  there exists an index $j\in Inc_i$ for which there does not exist any $i$-feasible index $k\in Inc_i$, $k\neq j$.
Hence  for any $k\in K=Inc_i\setminus\{j\}$ there exists a lower $(j,k)$-test function $\phi_k$ of $u$ at $x$ with
 \begin{equation}\label{4:lemma12}
   H^j(x,\partial_j\phi_k(x))<0.
 \end{equation}
By Remark \ref{1:r2}(i) we have
 \begin{equation}\label{4:lemma12b}
 H^j(x,\partial_j\phi_k(x))= H^k(x,\partial_k\phi_k(x))<0.
 \end{equation}
It is not restrictive to assume that $u(x)=\phi_k(x)$ for any $k\in K$.
By adding a term of the form $-\alpha d(z,x)^2$ we may assume
that $u-\phi_k$ attains a strict minimum point at $x$. Hence by \eqref{1:H1} and \eqref{4:lemma12b} there exists
$r>0$ such that for all $k\in K$ and $z\in B_r(x)\setminus \{x\}$
\begin{equation}\label{4:lemma13}
\begin{split}
&u(z)-\phi_k(z)>0\\
& H^j(z,\partial_j\phi_k(z))<0,\quad H^k(z,\partial_k\phi_k(z))<0.
\end{split}
\end{equation}
Let $\xi>0$ be such that
\begin{equation}\label{4:lemma14}
    u(z)-\phi_k(z)>\xi \qquad \text{for all $k\in K$ and $z\in \partial B_r(x)$}.
\end{equation}
Define $\tilde \phi_k(z):=\phi_k(z)+\xi$ and $\tilde v: \{x\}\cup\bigcup_{k\in Inc_i}\bar e_k\to \R$ by
\[\tilde v(z):=\left\{
                  \begin{array}{ll}
                    \max_{k\in K}\tilde \phi_k(z), & \hbox{if $z\in\bar e_j$,} \\
                    \tilde \phi_k(z), & \hbox{if $z\in\bar e_k$, $k\in K$.}
                  \end{array}
                \right.
\]
We claim that $\tilde v$ is a  subsolution of \eqref{HJ} in $B_r(x)$.\\

\emph{Case 1:} Consider  $z\in B_r(x)\cap e_l$. If $l\in K$, then $\tilde v(z)=\tilde \phi_l(z)$ and the claim follows by
\eqref{4:lemma13}. If $l=j$, by \eqref{4:lemma13} we have
\[H^k(z,\partial_k\tilde\phi_k(z))<0\quad\forall k\in K\]
and the subsolution condition follows by Proposition \ref{1:prop1}.\\

\emph{Case 2:} Consider $z=x$. First assume $l,m\in K$, $l\neq m$, and let $\psi$ be an upper $(l,m)$-test function of $\tilde v$ at $x$.
Set
\begin{equation*}
    d_l:= a_{il}\partial_l \psi(x),\,d_m:= a_{im}\partial_m \psi(x), \,\eta_l:= a_{il}\partial_l \tilde\phi_l(x),\,\eta_m:= a_{im}\partial_m \tilde\phi_m(x).
\end{equation*}
As $\psi$ is $(l,m)$-differentiable at $x$, we have $d_l+d_m=0$. If $d_l\le 0$, we have $d_l\ge \eta_l$ by the definition of $\tilde v$ and by the fact that $\tilde v- \psi$ attains a local maximum at $z=x$. Hence $|d_l|\le |\eta_l|$. Similarly, if $d_m\le 0$, we have $d_m\ge \eta_m$, implying
$|d_m|\le |\eta_m|$. We therefore conclude that
\[|\pd_l\psi(x)|=|\pd_m\psi(x)|\le \max\{|\pd_l\tilde\phi_l(x)|,\,|\pd_m\tilde\phi_m(x)|\}.\]
By the assumptions on $H$ (see \eqref{1:H1}-\eqref{1:H5}) the function $h:\R\to\R$, $p\mapsto H^s(x,p)$ for $s\in Inc_i$ is independent of $s$, symmetric
at $p=0$, and strictly increasing in $|p|$. Hence by \eqref{4:lemma12b}
\begin{align*}
H^l(x,\partial_l \psi (x))=h(\partial_l \psi (x))\le \max\{h(\pd_l\tilde\phi_l(x)),\,h(\pd_m\tilde\phi_m(x))\}=\\
\max\{H^l(x, \pd_l\tilde\phi_l(x)),\,H^m(x,\pd_m\tilde\phi_m(x))\}<0.
\end{align*}
Assume now that $\psi$ is an upper $(j,l)$-test function of $\tilde v$ at $x$ and set
\begin{equation*}
    d_j:= a_{ij}\partial_l \psi(x),\,d_l:= a_{il}\partial_l \psi(x), \,e_j:=\max_{k\in K} a_{ij}\partial_j \tilde\phi_k(x),\,\eta_l:= a_{il}\partial_l \tilde\phi_l(x).
\end{equation*}
As above, we have $|d_j|\le |e_j|$ if $d_j\le 0$ and
$|d_m|\le |\eta_m|$ if $d_m\le 0$. Hence we get
\[|\pd_j\psi(x)|=|\pd_l\psi(x)|\le \max\{\max_{k\in K}|\pd_k\tilde\phi_k(x)|,\,|\pd_l\tilde\phi_l(x)|\},\]
and therefore by \eqref{4:lemma12b}
\begin{align*}
H^j(x,\partial_j \psi (x))=h(\partial_j \psi (x))\le \max\{\max_{k\in K}h(\pd_j\tilde\phi_k(x)),\,h(\pd_l\tilde\phi_l(x))\}=\\
\max\{\max_{k\in K}H^k(x, \pd_j\tilde\phi_l(x)),\,H^l(x,\pd_l\tilde\phi_l(x))\}<0.
\end{align*}
Hence   $\tilde v$ is a viscosity subsolution in $B_r(y)$. Define the function $v:\G\to \R$ by
\[v(z):=\left\{
          \begin{array}{ll}
            \max\{\tilde v(z),u(z)\}, & \hbox{if $z\in B_t(x)$,} \\
            u(z), & \hbox{if $z\in\G\setminus B_t(x)$.}
          \end{array}
        \right.
\]
By \eqref{4:lemma14}, $v$ is continuous, $v=u$ outside $B_r(x)$, and $v$ is a subsolution of \eqref{HJ} in $\G$.
Since $v(x)=\tilde v(x)>u(x)$, we get a contradiction to \eqref{2:prop41}.
\end{proof}
%\begin{corollary}
%The function $u\in\cS^-(\G)$ if and only if it satisfies \eqref{4:lemma31}.
%\end{corollary}

%%%%%%%%%%%%%%%%%%%%%%%%%%%%%%%%%%%%%%%%%%%%%%%%%%%
%                                                 %
%%%%%%%%%%%%%%%%%%%%%%%%%%%%%%%%%%%%%%%%%%%%%%%%%%%
\section{A comparison Theorem}\label{sec4}
This section is devoted to the proof of a comparison theorem for problem \eqref{HJ}.
\begin{theorem}\label{3:cor1}
Assume that there exists a closed subset $K\subset \G$ and a function $f\in C(\G)$ with $f(x)<0$ for all $x\in\G\setminus K$. Moreover, let $u$ be a subsolution of
\begin{equation}\label{HJsbis}
    H(x,D u)=f(x),\qquad x\in\G\setminus K,
\end{equation}
and let $v$ be a supersolution of \eqref{HJ} in $\G\setminus K$. If $u\le v$ on $\partial\G\cup K$, then $u\le v$ in $\bar\G$.
\end{theorem}
\begin{proof}
Assume by contradiction that there exists $z\in\G\setminus K$ such that
\begin{equation}\label{5:1}
    u(z)-v(z)=\max_{\bar\G}\{u-v\}=\delta>0.
\end{equation}
For $\e>0$ define $\Phi_\e:\bar\G\times\bar\G\to \R$ by
\[\Phi_\e(x,y):=u(x)-v(y)- \e^{-1}d(x,y)^2.\]
As $\Phi_\e$ is upper semicontinuous there exists a maximum point $(p_\e,q_\e)$ for $\Phi_\e$ in $\bar\G^2$. By
$\Phi_\e(z,z)\le \Phi_\e(p_\e,q_\e)$ we get
 \begin{equation}\label{5:2}
 \e^{-1}d(p_\e,q_\e)^2\le u(p_\e)-v(q_\e)-\delta,
 \end{equation}
whence
\begin{equation}\label{5:3}
    \lim_{\e\to 0} d(p_\e,q_\e)=0.
\end{equation}
By the compactness of $\bar \G$, there exists $\bar p\in\bar \G$ such that $p_\e,q_\e \to \bar p$. By \eqref{5:2} and the Lipschitz
continuity of $u$ (see \eqref{2:prop1}) we get
\begin{equation*}
 \e^{-1}d(p_\e,q_\e)^2\le u(p_\e)-u(q_\e)+u(q_\e)-v(q_\e)-\delta\le Ld(p_\e, q_\e)
 \end{equation*}
and  therefore
\begin{equation}\label{5:4}
\lim_{\e \to 0^+}\e^{-1}d(p_\e,q_\e)=0.
\end{equation}
Moreover, by \eqref{5:2}-\eqref{5:3} we have $\bar p\in\G\setminus K$
as well as $p_\e, q_\e\in\G\setminus K$ for a sufficiently small choice of $\e$. Next observe that it is possible to assume that there exists a unique  path
$\g=\g_\e$ of length $d(p_\e,q_\e)$ in $\G$ connecting $p_\e$ and $q_\e$ which runs through at most one vertex $v_i$, $i\in I$. We distinguish several cases
(for simplicity we set $p:=p_\e$, $q:=q_\e$).

%%%%%
\emph{Case 1:} There are indices $i\in I$ and $j,k\in Inc_i$ such that $p\in e_j$, $q \in e_k$ and such that $\g$ runs through $v_i$.
We observe that the functions $\phi_p(x):=\e^{-1}d(p,x)^2$ and $\phi_q(x):=\e^{-1}d(x,q)^2$ are differentiable at $q$ and $p$, respectively.
In fact, if
\begin{equation}\label{5:5}
\tilde p=\pi_j^{-1}(p),\quad\tilde q=\pi_k ^{-1}(q),
\end{equation}
we have
\begin{align*}
    \pd_j\phi^j_q(\tilde p)=\e^{-1}d(p,q)a_{ij},\quad  \pd_k\phi^k_p(\tilde q)=\e^{-1}d(p,q)a_{ik}.
\end{align*}
Observe that $u-\phi_q$ has a maximum point at $p$ and $v+\phi_p$ has a minimum point at $q$, whence
\begin{align*}
    &H^j( p,\pd_j\phi_q( p))=H^j(\tilde p,\e^{-1}d(p,q)a_{ij})\le f(p) \\
    &H^k(q,\pd_k\phi_p(q))=H^k(\tilde q,-\e^{-1}d(p,q)a_{ik})\ge 0.
\end{align*}
We denote by $\omega_m$, $m=j,k$, the modulus of continuity of $H^m$ with respect to $(x,p)\in\bar e_m\times \R$.
By \eqref{1:H4} and \eqref{1:H5}  there is some $\eta>0$ such that for sufficiently small $\e>0$ we have
\begin{align*}
  \eta\le -f(p)\le  H^k(\tilde q,-\e^{-1}d(p,q)a_{ik})-H^j(\tilde p,\e^{-1}d(p,q)a_{ij})\le\\
H^k(v_i,-\e^{-1}d(p,q)a_{ik})-H^j(v_i,\e^{-1}d(p,q)a_{ij})+\omega_k(d(v_i,q))+\omega_j(d(v_i,p))=\\
H^j(v_i,\e^{-1}d(p,q)a_{ik})-H^j(v_i,\e^{-1}d(p,q)a_{ij})+\omega_k(d(v_i,q))+\omega_j(d(v_i,p))\\
\le \omega_j(\e^{-1}d(p,q)(|a_{ij}|+|a_{ik}|))+\omega_k(d(v_i,q))+\omega_j(d(v_i,p)).
\end{align*}
By \eqref{5:4} we get a contradiction for $\e\to 0$.

%%%%%%
\emph{Case 2:} There are indices $i\in I$ and $j\in Inc_i$ such that $p\in e_j$ and $q=v_i$.
As $q\in\G$, we have  $i\in I_T$.
Setting $\phi_p$ and $\phi_q$ as above and using a notation similar to \eqref{5:5} we have
\begin{equation}\label{5:6}
    \pd_j\phi^j_p(\tilde q)=-\e^{-1}d(p,q)a_{ij},\quad  \pd_k\phi^k_p(\tilde q)=\e^{-1}d(p,q)a_{ik}\quad\text{for all $k\in Inc_i$, $k\neq j$.}
\end{equation}
Hence
\begin{equation}\label{5:7}
    a_{ij}\pd_j\phi^j_p(\tilde q)+a_{ik}\pd_k\phi^k_p(\tilde q)=(-a_{ij}^2+a_{ik}^2)\e^{-1}d(p,q)=0\quad\text{for all $k\in Inc_i$, $k\neq j$.}
\end{equation}
Thus $\phi_p$ is $(j,k)$-differentiable at $q$ for all $k\in Inc_i$, $k\neq j$. Moreover,
\[\pd_j\phi^j_q(\tilde p)=\e^{-1}d(p,q)a_{ij}. \]
Since $u-\phi_q$ has a maximum point at $p$, if follows
\begin{equation}\label{5:8}
    H^j( p,\pd_j\phi_q( p))=H^j(\tilde p,\e^{-1}d(p,q)a_{ij})\le f(p).
\end{equation}
Moreover, since $v+\phi_p$ has a minimum point at $q=v_i$, there is an $i$-feasible index $k_0\in Inc_i$, $k_0\neq j$, for $j$. By \eqref{5:7},
$\phi_p$ is $(j,k_0)$-differentiable, whence we obtain
\begin{equation}\label{5:9}
  H^j(q,\pd_j\phi_p(q))=H^j(\tilde q,-\e^{-1}d(p,q)a_{ik})\ge 0.
\end{equation}
Subtracting \eqref{5:8} from \eqref{5:9} we derive a contradiction as in case 1.

%%%%
\emph{Case 3:} There are indices $i\in I$ and $j\in Inc_i$ such that $p=v_i$ and $q\in e_j$. We proceed
 as in case 2, observing that the definition of subsolutions is less restrictive than the definition of supersolutions
and therefore no extra argument is required.

%%%%
\emph{Case 4:} There are indices $j\in J$  such that $p, q\in e_j$ and $p\neq q$. Setting $\phi_p$ and $\phi_q$ as above and using a notation similar to \eqref{5:5} we have
\begin{equation}
    \pd_j\phi^j_q(  p)=-  \pd_j\phi^j_p(q).
\end{equation}
Hence we have
\begin{align*}
    &H^j( p,\pd_j\phi_q( p))\le f(p), \\
    &H^j(q,-\pd_j\phi_p(q))=H^j(q,\pd_j\phi_q( p))\ge 0
\end{align*}
and we conclude as in the previous cases.

%%%%
\emph{Case 5:} We finally assume that $p=q$. Assume $p=q=v_i$ for $i\in I_T$ (the case $p,q\in e_j$ for $j\in J$ is similar). Then
 \[\pd_j\phi^j_q( \pi_j^{-1}(v_i))=  \pd_j\phi^j_p( \pi_j^{-1}(v_i))=0\]
for all $j\in Inc_i$. In particular for each choice of $j,k\in Inc_i$, both $\phi_q$ and  $-\phi_p$ are $(j,k)$-differentiable and
we get a contradiction as in the previous cases.
\end{proof}

%%%%%%%%%%%%%%%%%%%%%%%%%%%%%%%%%%%%%%%%%%%%%%%%%%%
%                                                 %
%%%%%%%%%%%%%%%%%%%%%%%%%%%%%%%%%%%%%%%%%%%%%%%%%%%

\section{Representation formula for viscosity solutions}\label{sec5}
In this section we give  a representation formula for the solution of the Dirichlet problem
\begin{align}
    H(x,Du)=0,\qquad &x\in \G\label{4:dir1},\\
    u=g,\qquad &x\in \partial\G\label{4:dir2}.
\end{align}
%%%
% %
%%%
In addition to \eqref{1:H1}-\eqref{1:H5} and \eqref{H4}, in this section we assume  that
\begin{equation}\label{4:1}
\begin{split}
&\text{there exist a closed (possibly empty) subset $K\subset \G$,
a differentiable} \\
&\text {function $\psi$, and   $h\in C(\G)$ with $h<0$ in $\Gamma\setminus K$ such that}\\
   &\qquad\qquad H(x,D \psi)\le h(x),\qquad x\in\G\setminus K,
    \end{split}
\end{equation}
i.e. $\psi$  is differentiable the sense of Definition \ref{1:def2}
and a strict subsolution  in  $\Gamma\setminus K$.
%%%%%%%%%%%%%
%           %
%%%%%%%%%%%%%

\begin{proposition}\label{6:prop1}
Let $g:\bar\G\to \R$ be a continuous function satisfying
\begin{equation}\label{4:comp}
    g(x)-g(y)\le S(y,x)\qquad\text{for any $x$, $y\in K\cup \partial \G$},
\end{equation}
where $S$ is the distance defined in \eqref{2:dist_int}.
Then the unique viscosity solution of \eqref{4:dir1}--\eqref{4:dir2} is given by
\[u(x):=\min\{g(y)+S(y,x):\, y\in K\cup \partial \G\}.\]
\end{proposition}
\begin{proof}
By Proposition \ref{2:prop3} $u$ is a solution of \eqref{HJ}.
Observe that we have  $u(x)\neq g(x)$ for  $x\in K\cup \partial \G$ if and only if there is some $z\in   K\cup \partial \G$
such that $g(x)>S(z,x)+g(z)$. However, this is ruled out by assumption \eqref{4:comp}. Hence $u$ is a solution of \eqref{4:dir1}-\eqref{4:dir2}.
\\
Assume that there exists another solution $v$ of \eqref{4:dir1}-\eqref{4:dir2}. For $\theta\in (0,1)$ define
$u_\theta :=\theta u+(1-\theta)\psi$, where $\psi$ as in \eqref{4:1}. By adding a constant it is not restrictive to assume that $\psi$ is sufficiently small in such a way that
 \begin{equation}\label{4:2}
 u_\theta(x)\le u(x),\qquad x\in\bar\G.
 \end{equation}
 First, let $x\in e_j\cap (\Gamma\setminus K)$  for some $j\in J$ and let $\phi$ be  an upper test
function of $u$ at $x$. Setting $\phi_\theta:=\theta\phi +(1-\theta)\psi$ we obtain by means of convexity
 \begin{equation}\label{4:3}
    H^j(x, \pd_j \phi_\theta)\le \theta H^j(x,\pd_j\phi)+(1-\theta)H^j(x,\pd_j\psi)\le (1-\theta)h^j(x).
\end{equation}

  Secondly, assume that $x=v_i$ for some $i\in I_T$. Fix any two indices $j,k\in Inc_i$, $j\neq k$, and let
$\phi$ be an upper $(j,k)$-test function of $u$ at $x$. Setting $\phi_\theta:=\theta\phi +(1-\theta)\psi$ and observing that by definition \ref{1:def2}
$\phi_\theta$ is an upper $(j,k)$-test function of $u_\theta$ at $x$, we again obtain \eqref{4:3}. Hence $u_\theta$ is a viscosity subsolution
of
\[ H(x, \pd_j u)\le (1-\theta)h^j(x).\]
Applying theorem \ref{3:cor1} with $f=(1-\theta)h$ and \eqref{4:2}, it follows $u_\theta\le v$ for all $\theta\in (0,1)$. Letting $\theta$ tend to $1$ yields
$u\le v$. Exchanging the role of $u$ and $v$ we conclude that $u=v$ in $\bar\G$.
\end{proof}
%%%%%%%%%
%       %
%%%%%%%%%
%\begin{proposition}
%Let $g$ satisfy
%\begin{equation}\label{4:comp1}
%    g(x)-g(y)\le S(x,y)\qquad\text{for any $x$, $y\in \cA\cup I_B$}.
%\end{equation}
%Then the unique viscosity solution of \eqref{4:dir1}--\eqref{4:dir2} is given by
%\[u(x)=\min\{g(y)+S(x,y):\, y\in \cA\cup I_B\}\]
%\end{proposition}
\begin{remark}
For the problem $|D u|^2-\alpha(x)=0$ (see Remark \ref{1:rem1}) the existence of a strict subsolution follows
by setting $K:=\{x\in\G:\, \alpha(x)=0\}$, $\psi:=C$ for some suitable $C\in\R$, and $h(x):=-\alpha(x)$.
\end{remark}
If $g$ does not satisfy assumption \eqref{4:comp} we can still characterize $S$ as the maximal solution
of the problem.
\begin{proposition}
Let $g:\bar\G\to \R$ be a continuous function. Then
\[u(x):=\min\{g(y)+S(x,y):\, y\in K\cup \partial \G\}\]
is the maximal solution of \eqref{4:dir1}  among
the solutions $v$ of \eqref{4:dir1} which satisfy $v\le g$ on $ K\cup \partial \G$.
\end{proposition}
\begin{proof}
By Proposition \ref{2:prop3}, $u$ is a solution of \eqref{HJ}. If $v$ is a solution of \eqref{4:dir1}, then by \eqref{4:lemma31}
\begin{align*}
    v(x)\le v(y)+S(y,x)\le g(y)+S(y,x)\quad\text{for any $y\in K\cup \partial \G$},
\end{align*}
and therefore the statement follows by Theorem \ref{3:cor1}.
\end{proof}
\begin{remark}
As explained in the introduction, a motivation of our work comes from the shortest path problem on a network.
The case of a weighted graph studied in graph theory   fits in our framework. In fact it is sufficient to choose the function $\a_j$
in Remark \ref{1:rem1} in such a way that  its integral   along the edge $e_j$ is equal to the given weight. We will study this problem in more details in a forthcoming paper.
\end{remark}

%%%%%%%%%%%%%%%%%%%%%%%%%%%%%%%%%%%%%%%%%%%%%%%%%%%
%                                                 %
%%%%%%%%%%%%%%%%%%%%%%%%%%%%%%%%%%%%%%%%%%%%%%%%%%%

\begin{acknowledgement}
This work is based on earlier results contained in the Ph.D. thesis of the first author. He would like to express his gratitude to Prof. K. P. Hadeler  for the support
and the guidance during the completion of the thesis.
\end{acknowledgement}
\bibliographystyle{amsplain}

\end{document}